\documentclass[11pt]{article}

\usepackage{latexsym,amsfonts,amsmath,amsthm,graphics}
\usepackage{times}
\usepackage{epsfig}
\usepackage{color}
\usepackage{algorithm}
\usepackage[a4paper,dvips]{geometry}
\usepackage{subfigure}
\usepackage{here}

\usepackage[font=small,labelfont=bf]{caption}

\newtheorem{theorem}{{\bf Theorem}}

\newtheorem{lemma}{{\bf Lemma}}
\newtheorem*{corollary}{{\bf Corollary}}
\newtheorem{remark}{{\bf Remark}}

\newtheorem*{thm2}{{\bf Corollary*}}

\newcommand{\RR}{\mathbb{R}}

\newcommand{\Ss}{\mathbb{S}}

\newcommand{\BB}{\mathbb{B}}

\begin{document}

\title{ Uniformly distributed sequences in the orthogonal group and on the Grassmannian manifold 
}

\author{ Florian Pausinger \footnote{ IST Austria, Am Campus 1, A-3400 Klosterneuburg, Austria. E-Mail: \emph{florian.pausinger@ist.ac.at} }}

\date{}

\maketitle

\begin{abstract}
Quasi-Monte Carlo methods replaced classical Monte Carlo methods in many areas of numerical analysis over the last decades. The purpose of this paper is to extend quasi-Monte Carlo methods into a new direction.
We construct and implement a uniformly distributed sequence in the orthogonal group $O(n)$. From this sequence we obtain a uniformly distributed sequence on the Grassmannian manifold $G(n,k)$, which we use to approximate integral-geometric formulas. 
We show that our algorithm compares well with classical random constructions and, thus, motivate various directions for future research.
\\[6pt]
\\[0pt]
\\[0pt]
\end{abstract}



\section{Introduction}
\label{sec1}

Quasi-Monte Carlo methods replaced classical Monte Carlo methods in many areas of numerical analysis over the last decades. This is due to improved constructions of low discrepancy point sets and sequences which yield a fast decay of the occuring approximation errors thus outperforming random point sets in many practical situations; see \cite{DP10}.
The underlying pure mathematical framework, known as \emph{uniform distribution theory}, is very well developed in abstract settings. However, applications and concrete constructions of point sets are mainly studied and applied in the $n$-dimensional unit cube $[0,1)^n$ or on the unit sphere $\Ss^{n-1}$. These spaces have nice algebraic properties which allow a precise analysis of the appearing approximation errors.

The purpose of our paper is to extend quasi-Monte Carlo methods into a new direction. Let $G(n,k)$ denote the Grassmannian manifold, which is the space of all $k$-dimensional linear subspaces of $\RR^n$. Compact topological groups, especially the \emph{non-abelian} orthogonal group $O(n)$, and corresponding homogenous spaces, like $G(n,k)$, play an important role in many areas such as statistics, physics and integral geometry. While it is well-known how to generate uniform random elements in $O(n)$ (for an overview see \cite{Dia87} and references therein, especially \cite{Hei78,Ste80,TT82}), there are so far only existence results for the quasi-random setting in the form of uniformly distributed sequences in compact, non-abelian topological groups; see \cite{KuNi74}.

We present a mathematical framework that allows to construct uniformly distributed sequences in compact topological groups. In particular,
the main contributions of our paper are 
\begin{itemize}
\item[i)] the concrete construction of a uniformly distributed sequence in $O(n)$,
\item[ii)] an application of this sequence to the Grassmannian manifold $G(n,k)$, and
\item[iii)] an implementation and numerical comparison of our sequences to classical random constructions via the approximation of concrete integral geometric integrals.
\end{itemize}
On the theoretical side, we combine different results in order to extend the Monte Carlo construction to the quasi-random setting yielding the desired sequences.
On the practical side, we implement our results and give a proof of concept by showing that our quasi-random construction compares well with the random construction in concrete examples. 
Along the way, we encounter various interesting questions for future research aimed to extend the success of quasi-Monte Carlo methods into the direction of compact topological groups.

\paragraph{The subgroup algorithm. }
For every $n \geq 2$ the orthogonal group $O(n)$ and its normal subgroup $SO(n)$ are represented by orthogonal $n \times n$ matrices, either with determinant $\pm 1$ or only $+1$, which both form a group since they are closed under multiplication and taking inverses.
We construct a sequence by adapting the subgroup algorithm of Diaconis, Shahshahani \cite{Dia87} using a result of Veech \cite{Vee71}. Interestingly, this algorithm works for general (abelian or non-abelian) compact topological groups.

The idea of the subgroup algorithm is to consider a nested chain of compact (sub)groups (not necessarily normal).
We present the algorithm for our particular case in which we consider the chain
\begin{equation*}
O(n) \supset O(n-1) \supset O(n-2) \supset \ldots \supset O(2),
\end{equation*}
where $O(n-1)$ is the subgroup of $O(n)$ obtained via fixing the (unit column) vector $e_1 \in \RR^n$; that is $O(n-1)=\{ \Gamma \in O(n): \Gamma e_1 = e_1 \}$.
Consider the top two terms of the chain. 
The key lemma in \cite{Dia87} claims that the product of a uniform random element of $O(n-1)$, and a uniform random coset representative for $O(n-1)$ in $O(n)$ is a uniform random element in $O(n)$. 
We follow the topological convention and refer to the dimension of the manifold when speaking about the general unit sphere $\Ss^{n-1}$; consequently every $x \in \Ss^{n-1}$ is an $n$-dimensional vector.
It is well-known that $\Ss^{n-1}$ can be used to identify the cosets of $O(n-1)$ in $O(n)$, since $\Ss^{n-1} \cong O(n)/O(n-1)$. 
Thus, knowing how to find random elements in $O(n-1)$ and on $\Ss^{n-1}$ suffices to obtain a random element in $O(n)$ and hence random elements can be generated inductively.

We extend this idea to the quasi-random setting yielding the following theorem and its corollary which we apply to concrete integrals.
\begin{theorem} \label{thm1}
Given a uniformly distributed (ud) sequence on $\Ss^{n-1}$ and a ud sequence in $O(n-1)$, there exists an explicitly constructible ud sequence in $O(n)$.
\end{theorem}

\begin{corollary} \label{thm2}
Given a ud sequence in $O(n)$, there exists an explicitly constructible ud sequence on $G(n,k)$, for every $k$ with $1\leq k \leq n-1$.
\end{corollary}

\paragraph{Outline. } In Section \ref{sec2} we recall the concept of uniform distribution in compact topological groups and the result of Veech. We prove Theorem \ref{thm1} in Section \ref{sec3} and apply it to certain integrals over the Grassmannian in Section \ref{sec4}. 
In Section \ref{sec5} we explicitly construct the sequence and present numerical results, before we conclude our paper in Section \ref{sec6}.


\section{Preliminaries}
\label{sec2}

In this section, we recall important defintions and concepts about compact topological groups. We refer to the books of Hewitt \& Ross \cite{HeRo79} and Kuipers \& Niederreiter \cite[Chapter 4]{KuNi74} for further background on compact groups and for a detailed exposition of the theory of uniform distribution in such groups.

\paragraph{Compact groups and homogenous spaces. }
A \emph{compact topological group $G$} is a Hausdorff topological space which is also a group such that the group operations \emph{product} and \emph{inverse} are continuous functions.
A \emph{proper closed subgroup} is a proper closed subset $H$ of group elements of $G$ which is a group itself.
There is a natural topology in the quotient space $G/H$ such that the natural map $g \mapsto gH$ of $G$ onto $G/H$ is open and continuous.

Let $X$ be a topological space. A group $G$ \emph{acts} on $X$ if there is a map $G \times X \rightarrow X$, such that $(gh)x=g(hx)$ and $ex=x$ for all $g,h \in G$, $x \in X$ and for the identity element $e \in G$.
Some elements of a group acting on a space $X$ may fix a point. These group elements form a closed subgroup called the \emph{isotropy group}, 
defined by $G_x = \{ g \in G : gx=x \}$, $x \in X$.
A group action $G \times X \rightarrow X$ is \emph{transitive} if for every pair of elements $x, y \in X$ there is a group element such that $gx=y$.
Given a compact topological group $G$, a \emph{$G$-space} or \emph{homogenous space} is a space $X$ on which $G$ acts transitively. The space $X$ is then isomorphic to the left cosets of the isotropy group, $X \cong G/G_x$.
In particular, every compact topological group is a homogenous space and products of homogenous spaces are again homogenous.

There exists a unique non-negative regular normed Borel measure $\mu$ on a compact topological group $G$ which is left translation invariant; that is $\mu(gB)=\mu(B)$ for all $g \in G$ and all Borel sets $B \in \mathcal{B}(G)$. This measure is called the normed \emph{Haar measure} on $G$ with normalization $\mu(G)=1$. Because of the compactness of $G$ this measure is also right translation invariant and thus we call it \emph{invariant}.
Given a homogenous space $X$, there is a unique $G$-invariant Borel measure, $\rho$, on $X$ defined by $\rho(B) = \mu (\{g \in G : g x_0 \in B \})$, $B \in \mathcal{B}(X)$ with arbitrary, but fixed, $x_0 \in X$; see \cite[Theorem 13.1.5]{SchWei08}.

\paragraph{Uniform distribution. }
A sequence $(w_m)$ in a compact topological group $G$ is said to be \emph{uniformly distributed (ud) with respect to the Haar measure} in $G$ if whenever $U$ is an open set of $G$ whose boundary has measure 0, and $\mathbf{1}_U$ is the characteristic function of $U$, the equation
\begin{equation} \label{ud}
\underset{N \rightarrow \infty} \lim \frac{1}{N} \sum_{m=1}^N \mathbf{1}_U(w_m) = \mu(U)
\end{equation}
holds.
This definition is extended to a homogenous space $X$ if the Haar measure $\mu$ is replaced by the unique $G$-invariant Borel measure $\rho$ on $X$.
Importantly, it can be shown that the sequence $(w_m)$ is ud in $G$ (resp. $X$) if and only if 
\begin{equation} \label{ud2}
\underset{N \rightarrow \infty} \lim \frac{1}{N} \sum_{m=1}^N f(w_m) = \int_G f \,\, d\mu,
\end{equation}
holds for all complex-valued, continuous functions $f$ on $G$ (resp. $X$).

\paragraph{The method of Veech. }
We construct a uniformly distributed sequence via a theorem of Veech using normal numbers. Before recalling this theorem, we remark that Drmota, Morgenbesser \cite{DrMo10} recently presented a different construction method based on \emph{generalized Thue-Morse sequences}.

Veech \cite{Vee71} calls a sequence $(r_m)$ of positive integers \emph{uniformly distributed sequence generator (udsg)} if whenever $G$ is a compact group and $(z_m)$ a sequence in $G$ which is not contained in any proper closed subgroup, the \emph{generated} sequence $(w_m)$ with $w_m=z_{r_1} z_{r_2}\ldots z_{r_m}$ is uniformly distributed in $G$. In his remarkable paper, Veech not only shows that such sequence generators exist but gives also explicit constructions.

Fix an integer $b>1$, and let all real numbers $\alpha$, $0<\alpha<1$, be represented by their (unique) expansions to the base $b$, that is $\alpha=0.a_1 a_2 a_3 \ldots$, where the \emph{digits} $a_i$ are integers with $0\leq a_i <b$ for $i\geq 1$, and also $a_i < b-1$ for infinitely many $i$. 
Let $J=[\beta_1,\beta_2)$ be a subinterval of $[0,1)$. If $\alpha$ is $b$-normal (for a definition see \cite{KuNi74}), there exist infinitely many integers $q\geq 2$ such that $\alpha_q \in J$, where $\alpha_q=0.a_q a_{q+1} \ldots$. Veech arranges these integers in increasing order, forming a sequence $(q_m)$. Now let $(r_m)$ be the sequence of differences, that is $r_1=q_1 -1$, $r_2=q_2-q_1$, $\ldots$, then Veech proves
\begin{theorem}[Veech, \cite{Vee71}] \label{thm3}
The sequence $(r_m)$ is a uniformly distributed sequence generator if $\alpha$ is $b$-normal and if $J \subseteq [0,1]$ is an interval of length at least $1/b$.
\end{theorem}

As an example we mention \emph{Champernowne's number} obtained by concatenating the decimal representations of the natural numbers, that is
$$\alpha=0.123456789101112 \ldots.$$ 
This number is normal in base $10$. Now, let $(q_m)$ be the sequence of successive occurences of a $5$ in $\alpha$ (such that $q_1=5$, $q_2=21$, $\ldots$), then $r_1=q_1-1=4$, $r_2=q_2-q_1=16$, $\ldots$ defines a udsg; see \cite{lev99} for more involved constructions of normal numbers.


\section{Proof of Theorem \ref{thm1} }
\label{sec3}

The proof of Theorem \ref{thm1} is based on two lemmas which require additional definitions. Throughout this section all topological spaces are considered to be second countable.
Given two homogenous spaces $X$ and $Y$ with corresponding Borel measures $\rho_X$ and $\rho_Y$, we define the product space $X \times Y$ and the product measure $\rho_X \times \rho_Y$ in the usual way. The direct products of the open sets of $X$ and $Y$ form a basis of the product topology. Moreover, (for products of second countable spaces) the product $\sigma$-algebra is the Borel $\sigma$-algebra of the product topology, on which the product measure is induced by
$$(\rho_X \times \rho_Y) (B) = (\rho_X \times \rho_Y) (B_X \times B_Y) = \rho_X(B_X) \rho_Y(B_Y),$$
for every basis element $B=B_X \times B_Y$.

The first lemma shows how to bijectively map elements of a certain product of homogenous spaces to a related compact group. 

\begin{lemma} \label{lem2}
Let $G$ be a compact topological group, let $H \subset G$ be a proper closed subgroup and let $X=G/H$ be the space of cosets. Then there exists a bimeasurable, bijective map $T: G \rightarrow X \times H$. Moreover, the Haar measure $\mu$ on $G$ admits the decomposition $T(\mu)=\rho_X \times \rho_H$.
\end{lemma}

\begin{proof}[Sketch of proof]
We follow \cite{Dia87} in the definition of the map $T$. Let $\pi: G \rightarrow X$ be the map that assigns $g \in G$ to the coset containing $g$. To choose coset representatives, let $\phi: X \rightarrow G$ be a measurable inverse of $\pi$ (so $\pi \phi(x) = x$). The existence of $\phi$ under our assumptions follows from \cite[Theorem 1]{Bon76}. Define $T: G \rightarrow X \times H$ by
\begin{equation*}
T(g) = ( \pi(g), (\phi \pi (g) )^{-1} g).
\end{equation*}
This map is shown to be bimeasurable and bijective with inverse
\begin{equation*}
T^{-1}(x,h) = \phi(x) h.
\end{equation*}
Let $\mu$, $\rho_X$ and $\rho_H$ be invariant measures on $G$, $X$ and $H$ normalized so that each space has total mass 1. Then it follows from the definition of invariant measures and the product decomposition defined by $T$ that $T(\mu)=\rho_X \times \rho_H$; see \cite[Lemma 4.1]{Dia87}.
\end{proof}

The bijective map $T$ is not necessarily continuous and thus it seems in general difficult to see directly that it preserves the uniform distribution of a sequence in $X \times H$ which is mapped to $G$. However, the map $T$ can be used to obtain a sequence in $G$ that satisfies the assumptions of the Theorem of Veech as shown in the following lemma. Therefore we need one final definition.

Consider two sequences $(x_m)$ and $(y_m)$ in the homogenous spaces $X$ and $Y$. We construct a sequence $(u_m)$ in $X \times Y$ by combining the sequences $((x_m, e_Y))$ and $((e_X,y_m))$, with $e_X, e_Y$ being the neutral elements in $X$ and $Y$, in such a way that its first $k^2$ elements are just all possible pairs of $(x_i, y_j)$ with $1 \leq i \leq k$ and $1\leq j \leq k$. Specifically, we define $u_m$ by taking the unique integer $k\geq 1$ with $(k-1)^2 < m \leq k^2$, and setting $u_m = (x_k, y_i)$ if $m=(k-1)^2 + 2i -1$, and $u_m=(x_i, y_k)$ if $m=(k-1)^2 + 2i$. Thus, the first terms of the sequence $(u_m)$ are 
$$(x_1, y_1), (x_2, y_1), (x_1, y_2), (x_2, y_2), (x_3, y_1), (x_1, y_3), (x_3, y_2), (x_2, y_3), (x_3, y_3), \ldots$$
The sequence $(u_m)$ is called the \emph{convolution} of the sequences $((x_m, e_Y))$ and $((e_X,y_m))$, and is denoted by $(x_m) \ast (y_m)$; see also \cite{KuNi74}. 

\begin{remark} \label{rem1}
It can be shown that this construction preserves uniform distribution. That is, the sequence $(u_m)=(x_i) \ast (y_j)$ is uniformly distributed in the homogenous space $X \times Y$ if $(x_i)$ is ud in $X$ and $(y_j)$ is ud in $Y$.
\end{remark}

\begin{lemma} \label{lem3}
Let $H$ be a proper closed subgroup of $G$, which is not contained in any other proper closed subgroup of $G$, with $X=G/H$. Let $(x_m)$ and $(h_m)$ be ud in $X$ resp. $H$, and let $(u_m)=(x_i) \ast (h_j)$ be a sequence in $X \times H$.
Then the sequence $(T^{-1}(u_m))$ is not contained in any proper closed subgroup of $G$.
\end{lemma}

\begin{proof}
Our goal is to show that for every proper closed subgroup of $G$ there exists an element $u_m$ such that $T^{-1}(u_m)$ is not contained in this subgroup.

First, choose an arbitrary element $u_m=(x_i,h_j)$. If $T^{-1}(u_m)$ is contained in $H$, we know that $\phi(x_i)$ is in $H$ and, therefore, since $\pi \phi (x_i) = x_i$, we get that $x_i=eH \in X$ with $e$ being the neutral element. 
The left cosets $gH$ of the compact subgroup $H$ partition $G$ and are in bijection (via left multiplication) with each other, thus having the same measure. 
Hence, for every coset $gH$ in $X$ we can choose an open set $A \subset X$ with $\rho_X(A) > 0$ that does not contain $gH$; this is especially true for $g=e$. Since $(x_m)$ is ud in $X$, $ \underset{N\rightarrow \infty} \lim 1/N \sum_{m=1}^{N} \mathbf{1}_U(x_m) = \rho_X(U),$
holds for all open subsets $U$ of $X$ and therefore also for $A$. Since $\rho_X(A) >0$, there exists an element $x_{\tilde{m}}=\tilde{x}$ in $A$. By modifying $u_m$ to $(\tilde{x},h_j)$ we obtain an element that is not mapped to $H$ and exists because of the definition of convolution.

Second, we fix an arbitrary proper closed subgroup $F$ with $H \not \subset F$ and choose again an arbitrary element $u_m=(x_i, h_j)$. If $T^{-1}(u_m)=\phi(x_i) h_j$ is in $F$, we know that $\phi(x)$ is in $F h_j^{-1}$ (which is a right coset of $F$). Let $H':=H \cap F$. This is a closed subgroup of $H$ such that the natural map $\alpha: H \rightarrow H / H'$ is continuous and open. 
For the coset $h_j H'$ in $H/H'$ we can again choose an open set in $H/H'$ with positive measure that does not contain $h_j H'$ and whose $\alpha$ preimage $B$ is open in $H$ with $\rho_H(B)>0$.
Since $(h_m)$ is ud in $H$, we know that $ \underset{N\rightarrow \infty} \lim 1/N \sum_{m=1}^{N} \mathbf{1}_U(h_m) = \rho_H(U),$
holds for all open subsets $U$ of $H$ and therefore also for $B$. This ensures the existence of an element $h_{\tilde{m}}=\tilde{h} \in B$, since $B$ has positive measure. Taking this $\tilde{h}$ we obtain the element $\tilde{u}=(x_i, \tilde{h})$ which exists again because of the definition of convolution. However, for this element we know that $T^{-1}(u)=\phi(x_i) \tilde{h}$ is not in $F$.
\end{proof}

To turn to our particular case we follow \cite{Dia87} and let $G=O(n)$ with $H=O(n-1)=\{ \Gamma \in O(n): \Gamma e_1 = e_1 \}$. Coset representatives for $O(n-1)$ in $O(n)$ can be specified by saying where $e_1$ goes. Thus, the coset space is identified with $X=\Ss^{n-1}\cong O(n)/O(n-1)$. Then $\pi(\Gamma) = \Gamma e_1$. 
Let $I$ denote the identity matrix and $v^t$ the transpose of $v$. The map
$$\phi(x)=\left\{
\begin{array}{ll}
I & \mbox{ if } x = e_1, \\
I - 2 v v^t / c, & \text{ if } x \neq e_1, \mbox{ with } v=-x + e_1, c=v^t v,
\end{array}\right.$$
is a measurable inverse of $\pi$ that is continuous except at $e_1$ (there is no continuous choice of coset representatives).
By \cite[Lemma 5]{MoSa43} the subgroup $O(n-1)$ is not contained in any other proper subgroup of $O(n)$. Thus, given ud sequences in $\Ss^{n-1}$ and $O(n-1)$, we apply Theorem \ref{thm3} to the sequence $T^{-1}(u_m)$ and obtain a uniformly distributed sequence in $O(n)$; see also Remark \ref{rem2}.
As for the base case $O(2)$, it suffices to pick uniformly distributed angles $\phi_m$ form the interval $[0,2\pi)$ together with ud elements $t_m$ from the set $\{-1,1\}$. Then
$$\left(\begin{matrix}
\cos(\phi_m) & \sin(\phi_m) \\
-t_m \sin(\phi_m) & t_m \cos(\phi_m)
\end{matrix} \right)$$
yields a uniformly distributed sequence in $O(2)$.


\section{Application to the Grassmannian}
\label{sec4}

In this section, we show a potential application of our ud sequence in $O(n)$. In the first paragraph we apply Theorem \ref{thm1} to obtain a ud sequence on the Grassmannian manifold $G(n,k)$. In the second paragraph, we apply this sequence to a concrete integral. We refer to the books of Schneider \cite{Sch14} and Schneider \& Weil \cite{SchWei08} for more details on convex and integral geometry.

\paragraph{Ud sequence on the Grassmannian. } 
It is well-known that the Grassmannian manifold $G(n,k) = O(n) / O(n-k) \times O(k)$ is a homogenous space on which the orthogonal group acts transitively.
The natural operation of $O(n)$ on $G(n,k)$ is given by $(\Gamma, L) \mapsto \Gamma L$, which is simply the image of $L$ under $\Gamma$. 
To get a topology on $G(n,k)$ the surjective (but not injective) function
\begin{equation*}
\beta_k: O(n) \rightarrow G(n,k), \,\,\,\,\, \Gamma \mapsto \Gamma L_k,
\end{equation*}
is introduced, in which $L_k$ is an arbitrary, but fixed element of $G(n,k)$.
Then $G(n,k)$ is endowed with the finest topology for which $\beta_k$ is continuous. 
Thus, the preimage $\beta_k^{-1}(A)$ of every open set $A \subseteq G(n,k)$ is open.
Moreover, as noted in Section \ref{sec2}, there is a unique Haar measure $\rho=\mu_k$ on $G(n,k)$, normalized by $\mu_k(G(n,k))=1$.
Letting $\mu$ be the measure on $O(n)$, $\mu_k$ is the \emph{image measure} of $\mu$ under the mapping $\beta_k$, which means
\begin{equation*} \label{imageM}
\mu_k (A) = \mu( \{ \Gamma \in O(n) : \Gamma L_k \in A\} ) = \mu(\beta_k^{-1}(A) ).
\end{equation*}
This naturally leads to a concrete version of the corollary of Theorem \ref{thm1}; see also \cite[Remark 3]{DrMo10}.
\begin{thm2}
Let $(x_m)$ be ud in $O(n)$. Then $(y_m):=(\beta_k(x_m))$ is ud in $G(n,k)$.
\end{thm2}
\begin{proof}
Note that if $\mu_k(\partial A)=0$, then $\mu(\partial \beta_k^{-1} (A))=0$.
Moreover, we observe that
\begin{equation*}
\frac{1}{N} \sum_{m=1}^N \mathbf{1}_A(y_m) - \mu_k(A) = \frac{1}{N} \sum_{m=1}^N \mathbf{1}_{\beta_k^{-1}(A)} (x_m) - \mu(\beta_k^{-1}(A)). 
\end{equation*}
Since $(x_m)$ is ud in $O(n)$ the right hand side converges to 0 as $N$ goes to $\infty$ for all open sets in $O(n)$ whose boundary has measure 0. It follows from the continuity of $\beta_k$ that the preimage of every open set $A$ in $G(n,k)$ is open in $O(n)$ and thus $(y_m)$ is ud in $G(n,k)$.
\end{proof}
Finally, to prepare for the next paragraph, we introduce $A(n,k)$ as the space of all $k$-dimensional affine subspaces of $\RR^n$, the \emph{affine Grassmannian}, on which there exists a unique motion invariant, normalized Haar measure $\nu_k$ with
$\nu_k(\{ E \in A(n,k) \mid E \cap \BB^n \neq \emptyset\}) = b_{n-k}$, with $b_{n-k}$ being the volume of the $(n-k)$-dimensional unit ball $\BB^{n-k}$.

\paragraph{Integral-geometric formulas. } 
Let $\mathcal{K} \subset \RR^n$ be a \emph{convex body}, that is a compact, convex set, and let $V_0, V_1, \ldots, V_n$ denote its intrinsic volumes, which are geometric functionals on the space, $\mathcal{K}^n$, of all compact bodies in $\RR^n$. This space can be made into a metric space using the Hausdorff metric. The \emph{volume}, $V_n$, the \emph{surface area}, $2 V_{n-1}$,
and the \emph{Euler characteristic}, $V_{0} = \chi$,
are often of special interest.
The intrinsic volumes can be characterized by their properties,
namely that they are additive, motion invariant, and continuous.
Their importance is underlined by \emph{Hadwiger's Characterization Theorem},
which states that any additive, motion invariant, and continuous function
on $\mathcal{K}^n$ is a linear combination of the intrinsic volumes; see \cite{Had51, Had52}.

The famous \emph{Crofton formula} provides integral representations for the intrinsic volumes of a convex body.
In the following, when integrating with respect to the Lebesgue measure in $\RR^n$ we simply write
$d y$. 
For our example, we use a special case of the classical Crofton Formula:
\begin{equation} \label{intr}
  V_{n-k} (\mathcal{K})  =  c_{k,n} \cdot
                     \int_{E \in A(n,k)} \chi(\mathcal{K} \cap E) \ d \nu_k
\end{equation}
for $0 \leq k \leq n-1$, where $\mathcal{K} \in \mathcal{K}^n$ is a convex body in $\RR^n$,
$\chi (\mathcal{K} \cap E)$ is the Euler characteristic of the intersection,
and $c_{k,n} = {n \choose k} \frac{b_n}{b_k b_{n-k}}$.
Using \cite[Theorem 13.2.12]{SchWei08}, we can rewrite \eqref{intr} and obtain
\begin{equation} \label{intr2}
V_{n-k} (\mathcal{K})  =c_{k,n} \cdot
                     \int_{L \in G(n,k)} \int_{y \in L^{\perp}}
                     \chi \left( \mathcal{K} \cap (L+y) \right) \ d y \ d \mu_k,
\end{equation}
in which $L^{\perp} \in G(n,n-k)$ denotes the (unique) orthogonal complement of $L \in G(n,k)$ and $L+y$ denotes a translate of $L$.
Now we observe that the inner integral is simply the $(n-k)$-dimensional volume of the orthogonal projection of $\mathcal{K}$ onto $L^{\perp}$, denoted as $\mathcal{K} | L^{\perp}$.
The projection $\mathcal{K} | L^{\perp}$ is convex and varies continuously with $L$. Moreover, the volume functional is continuous on $\mathcal{K}^{n-k}$ and hence also its restriction to the subset consisting of all projections $\mathcal{K} | L^{\perp}$ for $L \in G(n,k)$. Thus, setting $f(L) =vol( \mathcal{K} | L^{\perp} )$ and using our ud sequence $(y_m)$ on $G(n,k)$, we get via \eqref{ud2}
\begin{equation}\label{intr3}
\underset{N \rightarrow \infty} \lim \left | \frac{1}{N} \sum_{m=1}^N f(y_m) - \int_{L \in G(n,k)} f(L) \ d \mu_k \right | = 0.
\end{equation}

\begin{remark} \label{rem2}
It would, of course, be interesting to have a quantified version of the convergence. 
In this context, we refer to the recent paper \cite{EdPa14}, in which this question is answered for the special case of integrating over $G(3,2)$ and looking at solid tubes instead of convex bodies. However, extending these results to general integrals over $G(n,k)$ poses intricate geometrical problems.
\end{remark}


\section{Implementation and numerical results}
\label{sec5}

This section contains all details needed to implement and test our sequence. We describe how to generate ud points on the sphere, outline our construction and show numerical results.

\paragraph{Ud sequences on the sphere. }
Distributing points on a hypersphere is a well studied problem. We refer to the classical paper of Pommerenke \cite{Pom59} for a construction of an infinite sequence and to Grabner, Klinger, Tichy \cite{GKT} for a quantitative analysis of various constructions and their use in numerical integration. 
To make our construction concrete, we recall Hlawka's appendix \cite{Hla82} to obtain a ud sequence on the sphere given a ud sequence in $[0,1]^n$; see \cite{DP10} for different constructions of ud sequences in $[0,1]^n$.
First, let $n=2k$ and let $(\alpha_m)$ be a ud sequence in $[0,1]^{2k}$. We write its $m$-th element as a row vector $(p_1(\alpha_m), q_1(\alpha_m), \ldots, p_k(\alpha_m), q_k(\alpha_m))$ and use the Box-Muller transform \cite{BM58}, to obtain a vector $(\xi_1, \eta_1, \ldots, \xi_k, \eta_k) \in \RR^{2k}$ with
\begin{equation*}
\xi_i = \sqrt{- \log p_i(\alpha_m)} \cos 2 \pi q_i(\alpha_m), \,\,\,\, \eta_i= \sqrt{- \log p_i(\alpha_m)} \sin 2 \pi q_i(\alpha_m).
\end{equation*}
In a next step, this vector is normalized to
\begin{equation*}
\Phi(\alpha_m):= \left( \frac{\xi_1}{r}, \frac{\eta_1}{r}, \ldots , \frac{\xi_k}{r}, \frac{\eta_k}{r} \right),
\end{equation*}
with $r^2=\xi_1^2 + \eta_1^2 + \ldots + \xi_k^2 + \eta_k^2$, yielding a point on the sphere $\Ss^{n-1}$, such that the sequence $(\Phi(\alpha_m) )$ is uniformly distributed on $\Ss^{n-1}$.
Concerning odd dimensions, we can simply omit $\xi_1$ in the above construction and obtain a ud sequence on $\Ss^{n-2}$ in a similar fashion.

\paragraph{Constructing a sequence. }
Applying our theorem, it is enough to know how to obtain uniformly distributed sequences on $\Ss^{i-1}$, $i=1, \ldots, n$, to obtain a ud sequence in $O(n)$. 
Having sequences $(x_m)$ on $\Ss^{n-1}$ and $(y_m)$ in $O(n-1)$, we immediately obtain a sequence in $O(n)$ in 3 steps:
\begin{itemize}
\item[(1)] Form the convolution $(u_m)=(x_m) \ast (y_m)$ in $\Ss^{n-1} \times O(n-1)$.
\item[(2)] Map $(u_m)$ via $T^{-1}$ to $O(n)$.
\item[(3)] Use Champernowne's number as a uniformly distributed sequence generator to modify the sequence $T^{-1}(u_m)$.
\end{itemize}
Using the subgroup algorithm to generate a random element in $O(n)$ is an $\mathcal{O}(n^3)$ algorithm; for details see \cite{Dia87, Ste80}.
(Note that $n$ is just the size of the matrices and is independent of the number of generated points!) 
Our quasi-random approach requires an additional matrix multiplication in the last of the above steps and thus the complexity of our algorithm is $\mathcal{O}(n^4)$. 
However, since $n$ is in general fixed and rather small this does not make any significant difference in practice.

\begin{remark} \label{rem3}
From a practical point of view it is interesting to note that the map $\phi(x)$ is almost continuous in our particular case. One can therefore safely omit the third of the above steps and still obtain a quasi-random sequence with good uniform distribution properties as long as the convolution of the two sequences is ud in the product space; see Remark \ref{rem1}.
\end{remark}

\paragraph{Numerical results}
To test our ud sequences $(y_m)$ we approximate different Crofton formulas via \eqref{intr3}. More precisely, for a given $n$-dimensional convex body $\mathcal{K}$, and a fixed $1\leq k \leq n$ we evaluate the function $f(L) =vol( \mathcal{K} | L^{\perp} )$, $L \in G(n,k)$, $N$-times and compute its mean 
$$I_{n,k}^N=\frac{1}{N} \sum_{m=1}^N f(y_m).$$
We implemented three different versions of this approximation. First, we computed random elements on $G(n,k)$, then we approximated the integral with our sequences following the above three steps and finally we computed quasi-random elements according to Remark \ref{rem2}.
We tested the three different implementations on various multi-dimensional convex polytopes and summarize our results for the first two algorithms in Table \ref{table1}, Table \ref{table2} and Figure \ref{fig1}. We note that the third implementation behaves always similar to the random approximation.
To generate the quasi-random elements we used \emph{scrambled Halton sequences}; see \cite{DP10,Fau09} for definitions and good choices of parameters.
Our test polytopes are as follows. In $\RR^3$, we use the unit cube (3-cube), the standard simplex (3-simplex) and Kirkman's icosahedron (K-icosahedron). The last polytope is given as the convex hull of 
$$(\pm 9, \pm 6, \pm 6), (\pm 12, \pm 4, 0), (0, \pm 12, \pm 8), (\pm 6, 0, \pm 12),$$
see \cite{Fe12} for more information about this interesting polytope.
In $\RR^4$, we used again the unit cube (4-cube) and the standard simplex (4-simplex). 
Furthermore, we construced two random polytopes by sampling 50 random points on $\Ss^2$ resp. $\Ss^3$ and taking their convex hulls.

\begin{table}[h!] 
\centering
{\small
\begin{tabular}{ l | c |  c | c c c}
polytope &  algo & $\# ver$ & $N=10$ & 100 & 1000  \\ 
\hline  
3-simplex			& 	r	& 4		& (0.547,\ 1.262)	& (0.592, 1.127) & (0.598,\ 1.107) \\ [0.1cm] 
					& 	qr	& 4		& (0.614,\ 1.140)	& (0.596,\ 1.134) & (0.597,\ 1.109) \\ [0.1cm] 
3-cube				& 	r	& 8		& (1.551,\ 1.450)	& (1.508,\ 1.516) & (1.519,\ 1.508) \\ [0.1cm] 
					& 	qr	& 8		& (1.473,\ 1.590)	& (1.532,\ 1.523) & (1.513,\ 1.506) \\ [0.1cm] 
K-icosahedron	& 	r	& 20	& (445.05,\ 25.45)	& (454.89,\ 24.89) & (454.29,\ 24.97) \\ [0.1cm] 
					& 	qr	& 20	& (459.92,\ 24.82)	& (455.20,\ 24.92) & (456.00,\ 25.01) \\ [0.1cm] 
r-polytope			& 	r	& 50	& (2.760, \ 1.937)	& (2.798, \ 1.910) & (2.785, \ 1.918) \\ [0.1cm] 
					& 	qr	& 50	& (2.801,\ 1.938)	& (2.768, \ 1.921) & (2.790, \ 1.921) \\ [0.1cm] 
					& 	r	& 150	& (3.009, \ 1.982)	& (3.018, \ 1.977) & (3.020, \ 1.975) \\ [0.1cm] 
					& 	qr	& 150	& (3.018,\ 1.975)	& (3.022, \ 1.975) & (3.020, \ 1.974) \\ [0.1cm]
\end{tabular}
}
\vspace{10pt}
\caption{Comparison of random and quasi-random approximation in $\RR^3$. The values in brackets show $I_{3,1}^N$ resp. $I_{3,2}^N$.} \label{table1}
\end{table}
\vspace{-5pt}
To explain the values we obtain, we recall the intrinsic volumes of the unit cube in $\RR^3$. The surface area, which is $2\cdot V_2$, of the 3-cube is 6. By \eqref{intr}, $2\cdot V_2(\text{3-cube}) = 4 \int_{L \in G(3,1)} f(L) \ d \mu_1$. Thus, we expect our algorithms to converge to $1.5$. Similarly, the integrated mean curvature, obtained as $\pi \cdot V_1$, of the 3-cube is $3 \pi$, and thus we again expect a value of $1.5$. Similar considerations allow to check the other values as well. In particular, note that the random polytopes approximate the corresponding spheres as the number of vertices increases. Since the surface area of $\Ss^2$ is $4\pi$ we expect $I_{3,1}^N$ to approximate $\pi=3.14\ldots$ from below, which can indeed be seen from our results.

\begin{table}[H] 
\centering
{\small
\begin{tabular}{ l | c |  c | c c c c }
polytope &  algo & $\# ver$ & $N=10$ & 100 & 1000 & 10000 \\ 
\hline  
4-simplex		& 	r	& 5		& 1.168	&  1.092 &  1.127 & 1.131\\ [0.05cm] 
				& 	qr	& 5		& 1.195	&  1.140 &  1.112 & 1.124\\ [0.05cm] 
4-cube			& 	r	& 16	& 1.657	& 1.682 & 1.676 & 1.672 \\ [0.05cm]
				& 	qr	& 16	& 1.732	& 1.665 & 1.664 & 1.665 \\ [0.05cm]
r-polytope		& 	r	& 50	& 1.826	& 1.826 & 1.818 & 1.820\\ [0.05cm]
				& 	qr	& 50	& 1.799	& 1.820 & 1.822 & 1.818 \\[0.05cm]
\end{tabular}
}
\vspace{10pt}
\caption{Comparison of random and quasi-random approximation in $\RR^4$. The values show $I_{4,3}^N$.} \label{table2}
\end{table}
\vspace{-5pt}

\begin{figure}[H]
\subfigure{\includegraphics[width=0.49\textwidth]{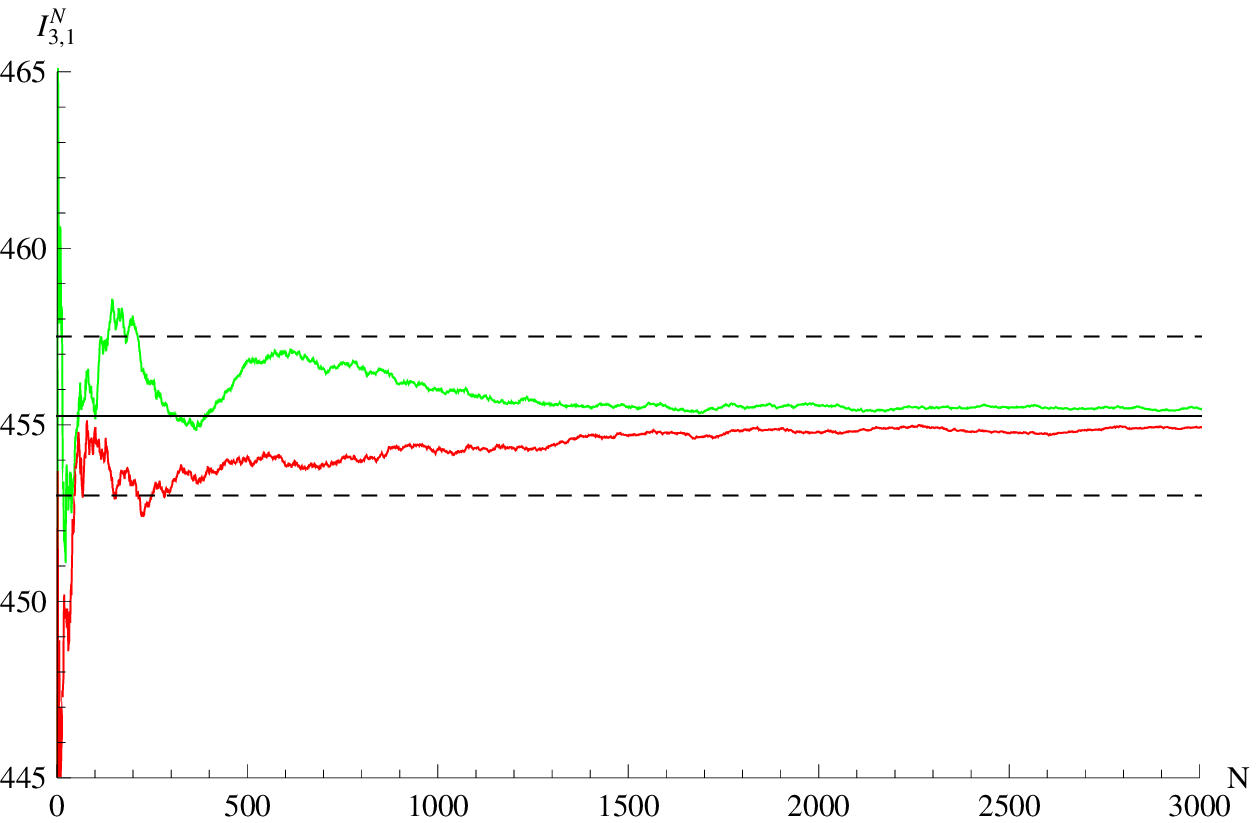}}\hfill
\subfigure{\includegraphics[width=0.49\textwidth]{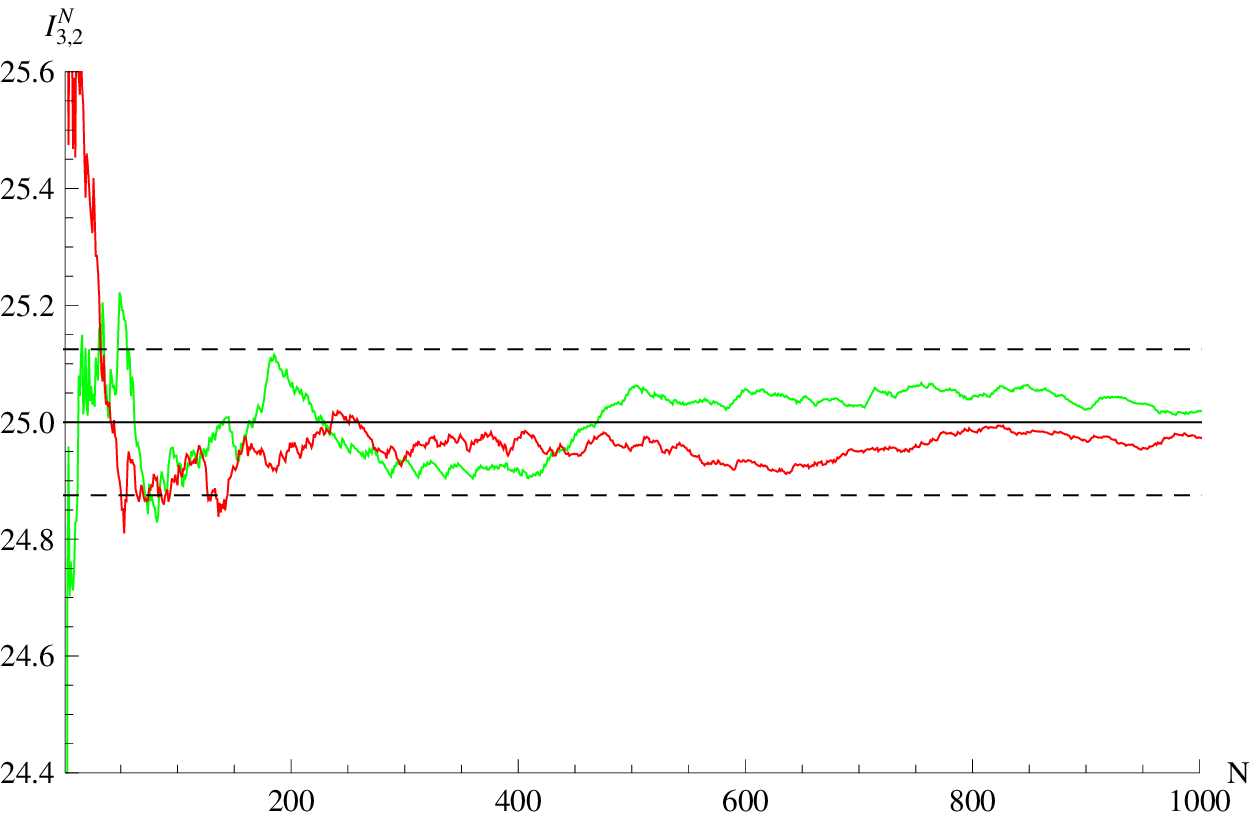}}
\caption{Comparison of random (red) and quasi-random (green) approximation of $I_{3,1}^N$ and $I_{3,2}^N$ for the K-icosahedron. The dashed lines indicate a deviation from the true value of $\pm 0.5\%$.} \label{fig1}
\end{figure}


\section{Concluding remarks}
\label{sec6}

We conclude with several related questions for future investigations.
How can the convergence to the uniform distribution be quantified in our particular setting? And more generally, is there a suitable concept of discrepancy in compact, non-abelian topological groups as it exists for sequenes in $[0,1)^n$ or on $\Ss^{n-1}$, which is amenable to a precise analysis. What are the general upper and lower bounds for the speed of convergence? We recall that the main problem in this context is that the map $T$ is in general not continuous.
Concerning non-continuous integrands, which appear in many integral geometric formulas, it is interesting to ask
which general concept of variation can be used to prove Koksma-Hlawka type results to bound the integration error when approximating such integrals? 
And from a practical point of view, which sequences outperform others significantly?

\subsection*{Acknowledgement}
The author thanks Harald Niederreiter and Robert Tichy for interesting discussions and Peter Hellekalek and Anne Marie Svane for carefully reading preliminary versions of this manuscript.


\end{document}